\documentclass[a4paper, 12pt]{article}
\usepackage{graphicx}
\usepackage{tikz}
\usepackage{tikz-cd}
\usepackage{xstring}
\usepackage[normalem]{ulem}
\usepackage{cancel}
\usetikzlibrary{matrix,arrows,decorations.pathmorphing}
\tikzset{commutative diagrams/.cd}
\usepackage{pst-node}
\usepackage[english]{babel}
\usepackage{geometry}
\usepackage{amssymb,amsmath,amsfonts,bbm,pifont,upgreek,bbold,accents}  
\usepackage[colorlinks=true]{hyperref}
\hypersetup{urlcolor=blue, citecolor=red}

\font\teneufm=eufm10
\font\seveneufm=eufm7
\font\fiveeufm=eufm5
\newfam\eufmfam
\textfont\eufmfam=\teneufm
\scriptfont\eufmfam=\seveneufm
\scriptscriptfont\eufmfam=\fiveeufm

\newcommand\beq[1]{\begin{equation}\label{#1} }
\newcommand{\eeq}{\end{equation} }

\newcommand\beqa[1]{\begin{eqnarray} \label{#1}}
\newcommand{\eeqa}{\end{eqnarray} }
\newcommand{\beqano}{\begin{eqnarray*} }
\newcommand{\eeqano}{\end{eqnarray*} }

\newtheorem{theorem}{Theorem}

\newtheorem{definition}[theorem]{Definition}
\newtheorem{proposition}[theorem]{Proposition}
\newtheorem{lemma}[theorem]{Lemma}%[section]
\newtheorem{sublemma}[theorem]{Sublemma}%[section]
\newtheorem{remark}[theorem]{Remark}%[section]
\newtheorem{notationalremark}[theorem]{Notational Remark}%[section]
\newtheorem{corollary}[theorem]{Corollary}%[section]
\newtheorem{assumption}[theorem]{Assumption}%[section]
\newtheorem{claim}[theorem]{Claim}%[section]
\newtheorem{tools}{$\negsp\negsp$}[subsection]

\newcommand\thm[1]{\begin{theorem}\label{#1}}
\newcommand\thmtwo[2]{\begin{theorem}[#1]\label{#2}}
\newcommand\ethm{\end{theorem} }
\newcommand\dfn[1]{\begin{definition}\label{#1} \rm}
\newcommand\dfntwo[2]{\begin{definition}[#1]\label{#2} \rm}
\newcommand\edfn{\end{definition} }
\newcommand\pro[1]{\begin{proposition}\label{#1}}
\newcommand\protwo[2]{\begin{proposition}[#1]\label{#2}}
\newcommand\epro{\end{proposition} }
\newcommand\lem[1]{\begin{lemma}\label{#1}}
\newcommand\lemtwo[2]{\begin{lemma}[#1]\label{#2}}
\newcommand\elem{\end{lemma} }
\newcommand\sublem[1]{\begin{sublemma}\label{#1}}
\newcommand\sublemtwo[2]{\begin{sublemma}[#1]\label{#2}}
\newcommand\esublem{\end{sublemma} }
\newcommand\rem[1]{\begin{remark}\label{#1} \rm}
\newcommand\erem{\end{remark} }
\newcommand\notrem[1]{\begin{notationalremark}\label{#1} \rm}
\newcommand\enotrem{\end{notationalremark} }
\newcommand\cor[1]{\begin{corollary}\label{#1}}
\newcommand\cortwo[2]{\begin{corollary}[#1]\label{#2}}
\newcommand\ecor{\end{corollary} }
\newcommand\asmp[1]{\begin{assumption}\label{#1}}
\newcommand\asmptwo[2]{\begin{assumption}[#1]\label{#2}}
\newcommand\easmp{\end{assumption} }
\newcommand\clm[1]{\begin{claim}\label{#1}}
\newcommand\eclm{\end{claim} }
\expandafter\chardef\csname pre amssym.def
at\endcsname=\the\catcode`\@
\catcode`\@=11
\def\undefine#1{\let#1\undefined}
\def\newsymbol#1#2#3#4#5{\let\next@\relax
 \ifnum#2=\@ne\let\next@\msafam@\else
 \ifnum#2=\tw@\let\next@\msbfam@\fi\fi
 \mathchardef#1="#3\next@#4#5}
\def\mathhexbox@#1#2#3{\relax
 \ifmmode\mathpalette{}{\m@th\mathchar"#1#2#3}%
 \else\leavevmode\hbox{$\m@th\mathchar"#1#2#3$}\fi}
\def\hexnumber@#1{\ifcase#1 0\or 1\or 2\or 3\or 4\or 5\or 6\or 7\or
8\or
 9\or A\or B\or C\or D\or E\or F\fi}
\ifcase\@ptsize
 \font\tenmsb=msbm10
 \font\sevenmsb=msbm7
 \font\fivemsb=msbm5
\or
 \font\tenmsb=msbm10 scaled \magstephalf
 \font\sevenmsb=msbm7 scaled \magstephalf
 \font\fivemsb=msbm5  scaled \magstephalf
\or
 \font\tenmsb=msbm10 scaled \magstep1
 \font\sevenmsb=msbm7 scaled \magstep1
 \font\fivemsb=msbm5 scaled \magstep1
\fi
\newfam\msbfam
\textfont\msbfam=\tenmsb
\scriptfont\msbfam=\sevenmsb
\scriptscriptfont\msbfam=\fivemsb
\edef\msbfam@{\hexnumber@\msbfam}
\def\Bbb#1{\fam\msbfam\relax#1}
\def\widehat#1{\setboxz@h{$\m@th#1$}%
 \ifdim\wdz@>\tw@ em\mathaccent"0\msbfam@5B{#1}%
 \else\mathaccent"0362{#1}\fi}
\def\widetilde#1{\setboxz@h{$\m@th#1$}%
 \ifdim\wdz@>\tw@ em\mathaccent"0\msbfam@5D{#1}%
 \else\mathaccent"0365{#1}\fi}

\def\RIfM@{\relax\ifmmode}
\def\nonmatherr@#1{\errmessage{\string#1\space allowed only in math mode}}
\def\Bbb{\RIfM@\expandafter\Bbb@\else
 \expandafter\nonmatherr@\expandafter\Bbb\fi}
\def\Bbb@#1{{\Bbb@@{#1}}}
\def\Bbb@@#1{\fam\msbfam\relax#1}
\def\setboxz@h{\setbox\z@\hbox}
\def\wdz@{\wd\z@}
\catcode`\@=\csname pre amssym.def at\endcsname
\newcommand{\negsp}{\hspace{-.09truecm}}  %%% equivalente a \!

\newcommand{\integer}{{\Bbb Z}   }
\newcommand{\complex}{{\Bbb C}   }

\newcommand{{\cH}}{{\cal H} }

 \title{\bf On Coble surfaces and their automorphisms}
 \date{}
 \author{Federico Pieroni\\
 University Roma Tre, Rome\\
{\small \tt federico.pieroni@uniroma3.it}}
 
\begin{document}
  \maketitle
\tableofcontents
\newpage

\begin{abstract}
In this article, we consider a smooth, complex Coble surface $X$, with a connected boundary curve $C \in |- 2 K_X|$. In~\cite{Pompilj1940}, Pompilj defined a biregular automorphism $T : X \to X$, and he claimed that the restriction $T|_C : C \to C$ equals the identity $\mathbb{1}_C$ on every Coble surface $X$. Coble~\cite{Coble1939} and Dolgachev~\cite{Dolg2019} proved that the family of Coble surfaces where this happens is a proper, locally closed subset of the moduli spaces of Coble surfaces.\\
Our goal is to show that this family has two components: one of them is made up of nodal Coble surfaces, and moreover the stronger equality $T = \mathbb{1}_X$ holds on every surface $X$ of this component, while the other component has codimension $3$ in the moduli space of Coble surfaces.
\end{abstract}

\section*{Introduction}

\addcontentsline{toc}{section}{Introduction}
%Coble surfaces were originally defined by A. Coble in an article of 1919, as the blow up of $\mathbb{P}^2$ of $10$ points $p_1, \ldots, p_{10}$ which are nodes for an irreducible sextic plane curve $\overline C \subset \mathbb{P}^2$. Nowadays, we use a different definition, which is purely stated in terms of divisors. 
In what follows, a surface is a smooth, projective, irreducible variety $X$ of dimension $2$ over the complex field $\complex$. In this setting, we say that $X$ is a Coble surface if it is rational and its canonical divisor $K_X$ satisfies $|- K_X| = \emptyset$ and $|- 2 K_X| = \{C\}$. This curve $C$ is known as the boundary of the surface $X$, and one can prove (see~\cite{CobleRatSurf2001}) that all its components are rational curves. The goal of this article is to study a particular family of Coble surfaces which are special from the point of view of their group of biregular automorphisms.
% We will  prove  a weaker version of this statement (see Theorem~\ref{basic}), starting from the hypothesis that $C$ is smooth.\\ 
The simplest example of a Coble surface was given by Coble in~\cite{OriginalCoble1919}, and it consists of the blow up of $\mathbb{P}^2$ at $10$ points $p_1, \ldots, p_{10}$ which are nodes for an irreducible sextic plane curve $\overline C \subset \mathbb{P}^2$. This construction also corresponds to the original definition given by Coble. We will see in Remark~\ref{act} that such a blow up always satisfies the axioms for Coble surface. In all this work, we will focus our attention on this model of Coble surface.\\
In this case, the boundary curve $C \subset X$ is given just by the strict transform of the sextic $\overline C$. In particular, in the Picard lattice $Pic (X)$ the class of $C$ is $C = - 2 K_X,$ and thus it is an invariant of the surface itself. As a consequence, every automorphism $\phi : X \to X$ carries $C$ to itself, that is, $\phi (C) = C$. It is natural then to ask if there are automorphisms $\phi$ of $X$ which pointwise fix the entire curve $C$. This question, known as Coble conjecture, is still an open problem. Coble expected that, for the general Coble surface, defined via a general $\overline C$ as above, there are no such automorphisms.\\
In~\cite{Pompilj1940}, Pompilj claimed that the conjecture was false: he constructed a non identical automorphism $T$ which is well defined on the general Coble surface, and he claimed that the equality \begin{eqnarray}\label{PC}T|_C = \mathbb{1}_C\end{eqnarray} was always true.\\
Coble~\cite{Coble1939}, and more recently Dolgachev~\cite{Dolg2019}, showed that equality~\eqref{PC} is generically false on a Coble surface, and it becomes true only along a suitable family defining a proper, locally closed subset in the moduli space of Coble surfaces.\\
The main result of this article is Theorem~\ref{final}. It states that this family has two components. Every member $X$ of the first component is nodal, that is, there exists an irreducible, smooth, rational curve $R \subset X$ such that $R^2 = -2$. Moreover, if equality~\eqref{PC} holds, then $T = \mathbb{1}_X$.\\
The second component has codimension $3$ in the moduli space of Coble surfaces. For any Coble surface $X$ in this component, we are able to give the explicit equation(see Lemma~\ref{codtre}) of a quintic surface $\overline X \subset \mathbb{P}^3$ which is a birational model for $X$.\\
The proof is structured in the following steps: In Section $1$ we recall some standard properties of the Del Pezzo surfaces of degree $1$ or $2$ and of the Coble surfaces, which we will use later.\\
Section $2$ is dedicated to a technical result: we take a Coble surface $X$ with irreducible boundary curve $C \in |- 2 K_X|$, together with three disjoint $(-1)$-curves $E_1, E_2, E_3 \subset X$. Then we describe the linear system defined by the divisor $H := C + E_1 + E_2 + E_3$. By adjunction formula, the three line bundles ${\cal O}_C (E_i)$ have degree $2$ on $C \simeq \mathbb{P}^1$. Moreover the image of the restriction map $$ H^0(\mathcal O_X(E_i)) \to H^0(\mathcal O_C(E_i)), \ i = 1,2,3, $$
is a $1$-dimensional subspace, generated by a vector $v_i$. The goal is to show that if $v_1, v_2, v_3$ are linearly dependent, then  $X$ admits $(-2)$-curves, whose classes are explicitly defined  in the Picard lattice $Pic (X)$. As a corollary, we will see how the linear system $|H|$ behaves on unnodal Coble surfaces. \\
In Section $3$ we prove our main result. As Coble and Dolgachev did, we first note that the automorphism $T$ decomposes as $$T = (i_3 \circ i_2 \circ i_1)^2,$$ where the $i_k$'s are three Bertini involutions. The proof of this fact is straight-forward, based on the very definition of $T$ given by Pompilj.\\
This means that the three restrictions $\sigma_k := (i_k)|_C : C \to C$, $k = 1,2,3$, are involutions on $C \simeq \mathbb{P}^1$, such that $$(\sigma_3 \circ \sigma_2 \circ \sigma_1)^2 = \mathbb{1}_C.$$ 
At this point, we distinguish two cases: if $\sigma_3 \circ \sigma_2 \circ \sigma_1 \ne \mathbb{1}_C$, Coble's strategy is to apply Pascal's Theorem~\ref{pascal}, to deduce a non - trivial condition of linear dependence on the degree two effective divisors $f_1, f_2, f_3$ which are respectively supported on the fixed points of $\sigma_1, \sigma_2, \sigma_3$. We write down explicitly this condition, showing that this implies that there are three disjoint $(-1)$-curves $E_1, E_2, E_3 \subset X$ whose restrictions $e_k  := E_k \cdot C$ belong to the same pencil of degree $2$ divisors of $C$. \\
At this point, by Section $2$ we know that $X$ admits $(-2)$-curves. The expression of these curves in $Pic (X)$ will give us informations on the mutual positions of the $10$ nodes, so that we will be able to state that the three Bertini involutions $i_1, i_2, i_3$ actually coincide, proving that $T = \mathbb{1}_X.$\\
In the other case we have $\sigma_3 \circ \sigma_2 \circ \sigma_1 = \mathbb{1}_C$, and we will still be able to recall the projective model described in Section $2$ to show that this happens in a $3$-codimensional family of Coble surfaces. 
\phantomsection
\addcontentsline{toc}{section}{Acknowledgements}
\paragraph{Acknowledgements}
The author is grateful to Prof. I. Dolgachev for his advice~\cite{Dolg2019} on Coble surfaces, and on the related problems addressed in this paper. In the same way the author is  grateful to his Ph.D. supervisor, Prof. A. Verra,  for countless inspiring and highlighting discussions on this subject. This work is partially based on the author's Ph.D. thesis~\cite{Pieroni2025}.
\newpage
\section{Basic properties}
\subsection{The Geiser birational involution}
In this subsection we will recall some standard facts about the Geiser involution on the Del Pezzo surface of degree $2$. For further details, see~\cite{Bayle2000}.\\
Let $p_1, \ldots, p_7$ be sufficiently general points in $\mathbb{P}^2$, and let $$Z := Bl_{p_1, \ldots, p_7} \mathbb{P}^2$$ be their  blow - up. This surface is called a Del Pezzo surface of degree $2$. Since $p_1, \ldots, p_7$ are sufficiently general, the linear system $|- K_Z|$ has no base locus and satisfies $h^0 ({\cal O}_Z (- K_Z)) = 3$. Moreover, it follows that $(- K_Z)^2 = 2$ and the anticanonical linear system $|- K_Z|$ induces a finite double covering $$q : Z \to \mathbb{P}^2.$$ The biregular involution associated to this double cover is called the Geiser involution of $Z$. We will denote such an involution by $$i_G : Z \to Z.$$ 
Every element ${\cal E} \in |- K_Z|$ is a curve of arithmetic genus $1$, moreover $q \vert \mathcal E$ is a finite double cover of a line $l \subset \mathbb{P}^2$. Hence, for a general $l$, $q \vert \mathcal E$ is branched at exactly four simple points and the branch curve of $q$ is a smooth quartic curve. On the curve ${\cal E}$, the Geiser involution $i_G$ takes the form $$z + i_G (z) \in |{\cal O}_{\cal E} (- K_Z)|, \quad \forall z \in {\cal E}.$$

\begin{proposition}\label{fixgeiser} Keeping the previous notation, a member ${\cal E} \in |- K_Z|$ is singular at $p$ if and only if $i_G (p) = p.$
\end{proposition}
{\it Proof:} Away from the fixed locus of $i_G$, the quotient map $q : Z \to \mathbb{P}^2$ is an \'etale local isomorphism. As a consequence, if $i_G (p) \ne p$ and
$p \in \mathcal E$, then $\mathcal E$ is smooth at $p$ since $l = q(\mathcal E)$ is smooth at $q(p)$.\\
Conversely, if $p$ is fixed by $i_G$, then its image $q (p) \in \mathbb{P}^2$ belongs to the quartic branch curve $B$ of $q$. If $l$ denotes the tangent line to $B$ at $q(p)$, then $\mathcal E = q^* (l)$ is  singular at $p$. $\square$

\subsection{The Bertini birational involution}
In this subsection we will recall some standard facts about the Bertini involution on the Del Pezzo surface of degree $1$. As in the previous subsection, we refer to~\cite{Bayle2000} for further details.\\
Let $p_1, \ldots, p_8 \in \mathbb{P}^2$ be $8$ points \it in general position\rm, and let $$Y := Bl_{p_1, \ldots, p_8} \mathbb{P}^2$$ be the blowing up of $\mathbb P^2$ at these points. This surface is a Del Pezzo surface of degree $1$. In particular its linear system $|- 2 K_Y|$ has no base points, and $h^0 ({\cal O}_Y (- 2 K_Y)) = 4.$ Let us consider the rational map $$f: Y \to \mathbb{P}^3,$$
defined by $|-2K_Y|$. Then, as is well known, $f$ is a finite double cover of a quadric cone $\overline S = f(Y) \subset \mathbb{P}^3$ and defines a biregular involution $$i_B : Y \to Y.$$
Let $\sigma: Y \to \mathbb P^2$ be the blowing up of $p_1 \dots p_8$. Then the birational involution
$$
\sigma \circ i_B \circ \sigma^{-1}: \mathbb P^2 \dasharrow \mathbb P^2
$$
is the classical Bertini involution associated to the points $p_1 \dots p_8$.\\
Each curve $D \in |- 2 K_Y|$ is a double cover of a hyperplane section of the cone $\overline S$, hence $$i_B (D) = D.$$ The arithmetic genus $p_a (D)$ equals $2$; when $D$ is smooth, the action of $i_B$ on $D$ is given by $$|{\cal O}_D (y + i_B (y))| = |{\cal O}_D (K_D)| \quad \forall y \in D.$$
The linear system $|- K_Y|$ satisfies $$h^0 ({\cal O}_Y (- K_Y)) = 2$$ and $$(- K_Y)^2 = 1.$$ Thus it consists of a pencil of elliptic curves, with one base point $p_b \in Y$. This point is fixed by the Bertini involution $i_B$, and its image in $\mathbb{P}^3$ is the vertex $v$ of the cone $\overline S$.\\ Every divisor ${\cal E} \in |- K_Y|$ is a double cover of a generating line of the cone $\overline S$, and consequently $$i_B ({\cal E}) = {\cal E}.$$ When ${\cal E}$ is smooth, the action of $i_B$ on ${\cal E}$ is given by $$|{\cal O}_{\cal E} (y + i_B (y))| = |{\cal O}_{\cal E} (2 p_b)| \quad \forall y \in {\cal E}.$$
The image under the multiplication map  
$$\mu: Sym^2 H^0 ({\cal O}_Y (- K_Y)) \to H^0 ({\cal O}_Y (- 2 K_Y))$$ has codimension $1$, and it corresponds to sections of $|- 2 K_Y|$ which split as a sum ${\cal E}_1 + {\cal E}_2$ of two sections of $|- K_Y|$. It also coincides with the subspace of sections of $|- 2 K_Y|$ passing through $p_b$.\\
The fixed locus of $i_B$ consists of the isolated point $p_b$ and a ramification curve $R \in |- 3 K_Y|$, which is a non - hyperelliptic curve of genus $4$. On the cone $\overline S$, the branch locus is given by their isomorphic images, namely the vertex $v$ and a curve $B \subset \overline S$ of degree $6$, obtained as complete intersection of $S$ with a cubic surface of $\mathbb{P}^3$.

\begin{proposition}\label{fix}
Let $Y$ be a Del Pezzo surface of degree $1$, with its Bertini involution $$i_B : Y \to Y.$$ Assume that $p \in Y$ is a double point for an irreducible, reduced divisor $D \in |- 2 K_Y|$.\\
Then $p$ is a non - isolated fixed point for $i$, and the tangent directions of $D$ at $p$ are swapped by $i$.
\end{proposition}
{\it Proof:} The curve $D$ is a double cover of an hyperplane section $H \cap S$ of the quotient quadric cone $S \subset \mathbb{P}^3$. Since $D$ is reduced and irreducible, the plane $H$ does not pass through the vertex $v \in S$, and hence $H \cap S$ is a smooth conic curve.\\
Since the curve $D$ is preserved by $i_B$, the point $i_B (p)$ is still a singularity of $D$. Moreover, the quotient map $q : Y \to S$ is a local etale isomorphism away from the fixed locus of $i_B$. Hence, if $i_B (p) \ne p$, then $q(p)$ would be a singularity of the quotient curve $q (D) = H \cap S,$ a contradiction. Thus $$i_B (p) = p.$$
Since $p_b$ is the only isolated fixed point in $Y$, this shows also that $p$ belongs to the fixed curve $R \subset Y$. Since $R$ is smooth, we can find an etale neighborhood $\Delta$ of $p$, with coordinates $(x, y) \in \Delta$, where $p = (0, 0)$ and the local equation of $R$ is given by $y = 0$. In these coordinates, the involution $i$ takes the form $$i (x, y) = (x, - y).$$ Let $$D = \{f (x, y) = 0\}$$ be a local equation for $D$ around $p$. Since $D$ is $i$-invariant, $f$ must be an eigenvector for $y \to - y$. If $f$ is a $(-1)$-eigenvector, all the monomials in $f$ have the form $\lambda_{h, k} x^h y^{2 k + 1}$. In particular $y$ divides $f$, so $D$ should share a common component with $R$, a contradiction with the irreducibility of $D$.\\
Hence $f$ is a $(+ 1)$-eigenvector for $i$, so all its monomials have the form $\lambda_{h, k} x^h y^{2 k}$. Moreover, $f$ must start with quadratic terms, since $D$ has multiplicity $2$ at $p$. This means: $$f (x, y) = a x^2 + b y^2 + {\cal I}_{(x, y)}^3, \quad (a, b) \ne (0, 0),$$ where ${\cal I}_{(x, y)}^3$ denotes terms which vanish at the origin with order at least $3$. If $\alpha, \beta \in \complex$ satisfy $\alpha^2 = a, \beta^2 = - b$, then $$f = (\alpha x + \beta y) (\alpha x - \beta y) + {\cal I}_{(x, y)}^3,$$ so clearly $i$ swaps the two tangent directions $\alpha x \pm \gamma y = 0$. $\square$

\begin{proposition}\label{liftgeiser}
Note that, if $Z$ is a Del Pezzo surface of degree $2$, and $p \in Z$ is fixed by the Geiser involution $i_G : Z \to Z$, then the Bertini involution $i_B$ of the Del Pezzo surface $Y := Bl_p Z$ is the lift of $i_G$.
\end{proposition}
{\it Proof:} We first claim that any two curves ${\cal E}_1, {\cal E}_2 \in |- K_Z|$ through $p$ satisfy: \begin{eqnarray}\label{sametg}{\cal E}_1 \cap {\cal E}_2 = 2 p.\end{eqnarray}
To prove this, we take the quotient map $$q : |- K_Z| :  Z \to \mathbb{P}^2.$$ Let $$b := q (p) \in \mathbb{P}^2.$$ Since $p$ is fixed by the Geiser involution, $b$ is a branch point for $q$. Then ${\cal E}_1, {\cal E}_2$ are double covers of two lines $l_1, l_2 \subset \mathbb{P}^2$ through $b$. We write the intersection $$l_1 \cap l_2 = b,$$ and since $b$ is a branch point, applying $q^*$ on both sides we get equality~\eqref{sametg}.\\
Let $\pi : Y \to Z$ be the blow - down, and let $E \subset Y$ be the exceptional curve over $p$. Equality~\eqref{sametg} implies that all curves in $|- K_Z|$ through $p$ have the same tangent direction at $p$, and hence there exists a point $p_b \in E$ which lies in all the strict transforms of these curves. But these strict transforms form the linear system $|- \pi^* K_Z - E| = |- K_Y|$, hence $p_b$ is exactly the base point of $|- K_Y|$.\\
Now we take a smooth curve ${\cal E'} \in |- K_Y|$, and we note that $${\cal E} := \pi ({\cal E'}) \in |- K_Z|$$ is sill smooth. Of course, the restriction $$\pi_{\cal E'} : {\cal E'} \to {\cal E}$$ is an isomorphism, and it satisfies $$\pi_{\cal E'} (p_b) = p.$$ Since $p$ is fixed by $i_G$, it is a Weierstrass point for the restricted double cover $q : {\cal E} \to \mathbb{P}^1$. Thus, on the curve elliptic curve ${\cal E}$ the Geiser involution $i_G$ is defined by: $$z + i_G (z) \in |{\cal O}_{\cal E} (2 p)| \quad \forall z \in {\cal E}.$$ 
If $i_G'$ denotes the lift of $i_G$ on ${\cal E'}$, we have then: $$y + i_G'(y) \in |{\cal O}_{\cal E'} (\pi_{\cal E'}^* (2 p))| = |{\cal O}_{\cal E'} (2 p_b)| \quad \forall y \in {\cal E'}.$$ This is the same relation defining the Bertini involution $i_B$ on ${\cal E'}$, thus these two involutions agree on ${\cal E'}$. Since the smooth members of $|- K_Y|$ are dense in $Y$, this concludes the proof. $\square$

\subsection{Coble surfaces}
\begin{definition}
We say that $X$ is a Coble surface if $X$ is  a  smooth, irreducible, projective rational surface, with $$|- K_X| = \emptyset$$ and $$|- 2 K_X| = \{C\}.$$ The unique anti-bicanonical curve $C$ is called the boundary curve of $X$. 
\end{definition}
The next example gives the simplest construction for a Coble surface, and it also corresponds to the first one which was historically known, see~\cite{OriginalCoble1919}.
\begin{remark}\label{act}
Assume that $\overline C \subset \mathbb{P}^2$ is an irreducible curve of degree $6$, with $10$ double points $p_1, \ldots, p_{10}$. Then the blow - up $$X := Bl_{p_1, \ldots, p_{10}} \mathbb{P}^2$$ is a Coble surface, and the anti - bicanonical curve $C$ is exactly the strict transform of $\overline C$.\\
Indeed, the Picard lattice $Pic (X)$ has a standard representation $$Pic (X) = \integer L \oplus \bigoplus_{i = 1}^{10} \integer E_i,$$ where $L \subset X$ is the pull - back of any line, and $E_i \subset X$ is the exceptional divisor associated to the point $p_i$. In this basis, the class of $C$ equals $$C = 6 L - 2 E_1 - \cdots - 2 E_{10} = - 2 K_X.$$ By adjunction formula $$p_a (\overline C) = 10,$$ and hence $$p_g (\overline C) = 0.$$ Consequently $$C \simeq \mathbb{P}^1.$$ Moreover, we can also compute the self - intersection $$C^2 = -4.$$ Consider the short exact sequence $$0 \to {\cal O}_X \to {\cal O}_X (C) \to {\cal O}_C (C) \to 0.$$ The last term equals $${\cal O}_C (C) \simeq {\cal O}_{\mathbb{P}^1} (C^2) = {\cal O}_{\mathbb{P}^1} (-4),$$ thus we find an isomorphism $$\complex = H^0 ({\cal O}_X) \simeq H^0 ({\cal O}_X (C)).$$ This proves that $|- 2 K_X|$ does not contain other divisors. Consequently, the anti - canonical system $|- K_X|$ must be empty: if it is not, an effective curve ${\cal E} \in |- K_X|$ would give a member $2 {\cal E} \in |- 2 K_X|$, contradicting the rigidity of $C$.
\end{remark}
Some general properties of Coble surfaces are given by the following Proposition:
\begin{proposition}\label{basic}
Let $\{C\} = |- 2 K_X|$ be the Coble curve in a Coble surface $X$, with irreducible decomposition $C = C_1 + \cdots + C_n$.\\
If the divisor $C$ is smooth, then:\\
i) the $C_i$'s are smooth rational curves,\\
ii) $C_i^2 = -4$,\\
iii) $K_X^2 = - n$.\\
iv) $h^1 ({\cal O}_X (- K_X)) = n - 1.$
\end{proposition}
{\it Proof:} $i)$ We use the short exact sequence \begin{eqnarray}\label{std} 0 \to {\cal O}_X (- C_1 - \cdots - C_n) \to {\cal O}_X \to \bigoplus_{i = 1}^n {\cal O}_{C_i} \to 0.\end{eqnarray} Since $X$ is rational, the long exact sequence gives: $$0 \to \bigoplus_{i = 1}^n H^1 ({\cal O}_{C_i}) \to H^2 ({\cal O}_X (- C_1 - \cdots - C_n)) \to \cdots$$ But by Serre's duality $$H^2 ({\cal O}_X (- C_1 - \cdots - C_n)) = H^2 ({\cal O}_X (2 K_X)) = H^0 ({\cal O}_X (- K_X)) = 0$$ which forces $$H^1 ({\cal O}_{C_i}) = 0.$$
$ii)$ By part $i)$ the $C_i$'s are smooth copies of $\mathbb{P}^1$, hence the adjunction formula gives: $$C_i^2 + C_i K_X = -2.$$ But from the other side, $$C_i^2 + C_i K_X = C_i^2 - \frac{1}{2} C_i (C_1 + \cdots + C_n) = \frac{1}{2} C_i^2,$$ hence we have the thesis.\\
$iii)$ We start from the equality $$- 2 K_X = C_1 + \cdots + C_n$$ and taking the squares of both sides, by part $ii)$, we have $$4 K_X^2 = -4 n.$$
$iv)$ The sequence~\eqref{std} induces a short exact sequence $$0 \to H^0 ({\cal O}_X) \to H^0 (\bigoplus_{i = 1}^n {\cal O}_{C_i}) \to H^1 ({\cal O}_X (- C_1 - \cdots - C_n)) \to 0.$$ This gives $$h^1 ({\cal O}_X (2 K_X)) = h^1 ({\cal O}_X (- C_1 - \cdots - C_n)) = n - 1,$$ and by Serre's duality, this is the same as $$h^1 ({\cal O}_X (- K_X)) = h^1 ({\cal O}_X (2 K_X)) = n - 1. \quad \square$$
\subsection{Involutions on $\mathbb{P}^1$}
In this subsection we will state some properties of the involutions on the projective line $\mathbb{P}^1$, since they will be crucial later.\\
Let us recall Pascal's theorem,~\cite{Shafarevich2013}. 
\begin{theorem}\label{pascal}
Let $p_1, \ldots, p_6$ be points of a smooth conic $ C \subset  \mathbb{P}^2$ and $\overline {p_ip_j}$ the line of $\mathbb P^2$ cutting on $C$ the divisor $p_i+p_j$, $1 \leq i,j \leq 6$.  Consider the lines $$ L_1 := \overline {p_1, p_2}, L_2 := \overline {p_3, p_4}, L_3 := \overline {p_5, p_6}, M_1 := \overline {p_4, p_5}, M_2 := \overline {p_6, p_1}, M_3 := \overline {p_2, p_3}, $$
then the three points $x_k := L_k \cap M_k$, for $k = 1, 2, 3$,  are collinear. 
\end{theorem}

\begin{proposition}\label{propuno}
Let $\sigma_1, \sigma_2, \sigma_3 : \mathbb{P}^1 \to \mathbb{P}^1$ be non identical involutions and $F_i \in H^0 ({\cal O}_{\mathbb{P}^1} (2))$ a section vanishing on the fixed points of $\sigma_i$, $i = 1,2,3$. Assume $\sigma_3 \circ \sigma_2 \circ \sigma_1 \ne \mathbb{1}_{\mathbb{P}^1}$ and $(\sigma_3 \circ \sigma_2 \circ \sigma_1)^2 = \mathbb{1}_{\mathbb{P}^1}$ then $F_1, F_2, F_3$ are linearly dependent.
\end{proposition}
{Proof:} Let embed $\nu: \mathbb{P}^1 \to \mathbb{P}^2$ as a smooth conic $C \subset \mathbb{P}^2$ via the complete linear system $|{\cal O}_{\mathbb{P}^1} (2)|$. The three involutions $\sigma_i: \mathbb{P}^1 \to \mathbb{P}^1$ extend then to three involutions $\Sigma_i : \mathbb{P}^2 \to \mathbb{P}^2$. Each $\Sigma_i$ has an isolated fixed point $x_i \in \mathbb{P}^2$ and a fixed line $L_i \subset \mathbb{P}^2$, which is the polar line of $C$ with center $x_i$. The three polynomials $F_i$ are the pull - backs $\nu^* (L_i)$, so we need to show that $L_1, L_2, L_3 \in H^0 ({\cal O}_{\mathbb{P}^2} (1))$ are linearly dependent. Let $[X_1, X_2, X_3]$ be projective coordinates on $\mathbb{P}^2$, and let $G \in H^0 ({\cal O}_{\mathbb{P}^2} (2))$ be the equation of the conic $C$ in these coordinates. Let $x_i = [a_{i, 1}, a_{i, 2}, a_{i, 3}]$ be the projective coordinates of each $x_i$. The linear equation of each $L_i$ is $$L_i = \sum_{j = 1}^3 a_{i, j} \frac{\partial G}{\partial X_j},$$ so we need to prove that $${\rm det\,} ((a_{i, j})_{i, j = 1}^3) = 0,$$ that is, the $x_i$'s are collinear. This follows from Pascal's Theorem~\ref{pascal}, choosing a point $p_1 \in C$, and defining $p_2 := \sigma_1 (p_1), p_3 := \sigma_2 (p_2), p_4 := \sigma_3 (p_3), p_5 := \sigma_1 (p_4), p_6 := \sigma_2 (p_5)$, and noting that $p_1 = \sigma_3 (p_6)$.\\
The hypothesis that $\sigma_3 \circ \sigma_2 \circ \sigma_1 \ne \mathbb{P}^1$ ensures that, for a generic choice of $p_1$, these are six distinct points. With the notations of Pascal's Theorem~\ref{pascal}, we have: $$x_1 = \overline {p_1, p_2} \cap \overline {p_4, p_5},$$ $$x_2 = \overline {p_2, p_3} \cap \overline {p_5, p_6},$$ $$x_3 = \overline {p_3, p_4} \cap \overline {p_6, p_1},$$ hence we find the thesis. $\square$\\\\
\begin{remark}
The hypothesis that $\sigma_3 \circ \sigma_2 \circ \sigma_1 \ne \mathbb{1}_{\mathbb{P}^1}$ cannot be removed. Otherwise, we find $$p_1 = p_4, p_2 = p_5, p_3 = p_6,$$ and thus the lines $$L_1 = M_1, \quad L_2 = M_2, \quad L_3 = M_3$$ coincide. In this case it does not make any sense to find the intersection points $L_i \cap M_i$.\\
An example of this behaviour is the following: if $t$ is an affine coordinate on $\mathbb{A}^1 \subset \mathbb{P}^1$, we take the involutions $$\sigma_1 (t) := -t,\quad \sigma_2 (t) := \frac{1}{t},\quad \sigma_3 (t) := - \frac{1}{t}.$$ Of course $\sigma_3 \circ \sigma_2 \circ \sigma_1 = \mathbb{1}_{\mathbb{P}^1}$.\\
The pair of fixed points of $\sigma_1$ is $0, \infty$, thus $$F_1 (U, V) = U V$$ in the projective coordinates $[U, V]$ over $\mathbb{P}^1$. Similarly, the pair of fixed points of $\sigma_2$ is $\pm 1$, thus $$F_2 (U, V) = U^2 - V^2.$$ Finally, the fixed points of $\sigma_3$ are $\pm i$, thus $$F_3 (U, V) = U^2 + V^2,$$ which is of course linearly independent from $F_1, F_2$.
\end{remark}
The following Proposition is necessary to deal with the case of three involutions satisfying $\sigma_3 \circ \sigma_2 \circ \sigma_1 = \mathbb{1}_{\mathbb{P}^1}$.
\begin{proposition}\label{casodue}
Let $A_1, A_2, A_3 \in H^0({\cal O}_{\mathbb{P}^1} (2))$ be three linearly polynomials of degree $2$ over $\mathbb{P}^1$, with corresponding divisors $D_1, D_2, D_3$. Let $$\sigma_1 : \mathbb{P}^1 \to \mathbb{P}^1$$ be the involution of the $g_2^1$ generated by $D_2, D_3$. Similarly, take $$\sigma_2 : \mathbb{P}^1 \to \mathbb{P}^1$$ the involution of the $g_2^1$ generated by $D_1, D_3$, and $$\sigma_3 : \mathbb{P}^1 \to \mathbb{P}^1$$ the involution of the $g_2^1$ generated by $D_1, D_2$.\\ Assume that $$\sigma_3 \circ \sigma_2 \circ \sigma_1 = \mathbb{1}_{\mathbb{P}^1}.$$ Then each $D_i$ coincides with the pair of fixed points of $\sigma_i$.
\end{proposition}
{\it Proof:} We have $$\sigma_1 = \sigma_3 \circ \sigma_2.$$ The divisor $D_1$ is preserved by both $\sigma_2, \sigma_3$ by construction. Thus also $\sigma_1$ preserves $D_1$. This implies that either $D_1$ is the pair of fixed points of $\sigma_1$, or it belongs to the $g_2^1$ generated by $D_2, D_3$. The second case is impossible, since $A_1, A_2, A_3 \in H^0 ({\cal O}_{\mathbb{P}^1} (2))$ are linearly independent. Hence we have the thesis.\\
For the other two pairs, we argue similarly on the cyclic permutations $\sigma_2 = \sigma_3 \circ \sigma_1$ and $\sigma_3 = \sigma_2 \circ \sigma_1$. $\square$
\begin{proposition}\label{propdue}
Let $A_1, A_2 \in H^0 ({\cal O}_{\mathbb{P}^1} (2))$ be two polynomials of degree $2$ without common factors, and let $\sigma: \mathbb{P}^1 \to \mathbb{P}^1$ be the involution associated to the $g_2^1$ generated by $A_1, A_2$. Then the pair of fixed points of $\sigma$ is cut by the Jacobian polynomial $J (A_1, A_2) = (\partial_U A_1) (\partial_V A_2) - (\partial_V A_1) (\partial_U A_2)$.
\end{proposition}
{\it Proof:} The quotient of $\mathbb{P}^1$ modulo $\sigma$ is the double cover $$\pi : [A_1, A_2]: \mathbb{P}^1 \to \mathbb{P}^1,$$ which ramifies at the fixed locus of $\sigma$. But on the ramification locus $\pi$ fails to be a local isomorphism, thus the Jacobian determinant $${\rm det\,} \left(\begin{array}{lr}
\displaystyle \frac{\partial A_1}{\partial U} & \displaystyle \frac{\partial A_1}{\partial V}\\\\
\displaystyle \frac{\partial A_2}{\partial U} & \displaystyle \frac{\partial A_2}{\partial V}
\end{array}
\right) = 0\,.\qquad \begin{array}{lr}
\\\\
\\\\
\end{array}
$$
must vanish. $\square$
\\

\newpage
\section{A projective model of a Coble surface}
The goal of this section is to find a sufficient condition for a Coble surface $X$ with irreducible boundary curve $C$ to contain $(-2)$-curves. We begin by  considering a set $E_1, E_2, E_3 \subset X$ of disjoint $(-1)$-curves, since we are interested to describe the linear system $H := C + E_1 + E_2 + E_3$.\\
The next lemma describes the most important properties of the linear system $|H|$ in view of our applications. 
\begin{lemma}\label{propzero}
Let $X$ be a Coble surface, with irreducible boundary curve $C \in |-~2~K_X|$, and let $E_1, E_2, E_3$ three disjoint $(-1)$-curves in $X$. Then:\\
i) The linear system $$H := C + E_1 + E_2 + E_3$$ has no base points or base components, and it has dimension $h^0 ({\cal O}_X (H)) = 4$. Hence $|H|$ induces a regular morphism $$f : X \to \mathbb{P}^3,$$ which is birational onto a quintic surface.\\
ii) The curve $C$ is embedded by $f$ as a smooth plane conic. The images of the $E_i$'s by $f$ are three pairwise distinct lines lying in the plane of $f(C)$.\\
iii) The three divisors ${\cal E}_i := E_i - K_X$ are effective and rigid, and their images are three distinct double lines of $f(X)$. Moreover, ${\cal E}_i^2 = 0$.\\
iv) The divisors $E_i + {\cal E}_i$ are still rigid.\\
v) If an irreducible curve $R$ is contracted by $f$, then $R$ is a smooth rational $(-2)$-curve.\\
vi) If the three restrictions $(E_i)|_C$ belong to the same $g_2^1$ on $C \simeq \mathbb{P}^1$, then $X$ contains $(-2)$-curves.
\end{lemma}
{\it Proof:} i) For any sequence of disjoint $(-1)$-curves $E_1, \ldots, E_r$ in $X$, the short exact sequence: $$0 \to {\cal O}_X (E_1 + \cdots + E_{r - 1}) \to {\cal O}_X (E_1 + \cdots + E_r) \to {\cal O}_{E_r} (E_1 + \cdots E_r) \to 0$$ induces isomorphisms $$H^0 ({\cal O}_X (E_1 + \cdots + E_{r - 1})) \simeq H^0 ({\cal O}_X (E_1 + \cdots + E_r))$$ and $$H^1 ({\cal O}_X (E_1 + \cdots + E_{r - 1})) \simeq H^1 ({\cal O}_X (E_1 + \cdots + E_r)).$$ Hence, by induction on $r$ we find: \begin{eqnarray}\label{hzero} h^0 ({\cal O}_X (E_1 + \cdots + E_r)) = h^0 ({\cal O}_X) = 1\end{eqnarray} and \begin{eqnarray}\label{hu} h^1 ({\cal O}_X (E_1 + \cdots + E_r)) = h^1 ({\cal O}_X) = 0.\end{eqnarray} Now we look at the short exact sequence: \begin{eqnarray}\label{inizio} 0 \to {\cal O}_X (- C) \to {\cal O}_X \to {\cal O}_C \to 0.\end{eqnarray} Taking the tensor with ${\cal O}_X (H)$ we find: $$0 \to {\cal O}_X (E_1 + E_2 + E_3) \to {\cal O}_X (H) \to {\cal O}_C (H) \to 0.$$ Due to equality~\eqref{hu}, it remains exact on global sections: $$0 \to H^0 ({\cal O}_X (E_1 + E_2 + E_3)) \to H^0 ({\cal O}_X (H)) \to H^0 ({\cal O}_C (H)) \to 0.$$ Note that, by Proposition~\ref{basic}: $${\cal O}_C (H) \simeq {\cal O}_{\mathbb{P}^1} (C H) = {\cal O}_{\mathbb{P}^1} (2).$$ Together with equality~\eqref{hzero}, this shows that $h^0 ({\cal O}_X (H)) = 4$ and that $C$ is not a base component of $H$, since $h^0 ({\cal O}_X (H - C)) = h^0 ({\cal O}_X (E_1 + E_2 + E_3)) = 1 < 4$. Moreover, the surjectivity of $H^0 ({\cal O}_X (H)) \to H^0 ({\cal O}_C (H))$ proves that $H$ has not base points on the curve $C$.\\
Now we tensor the short exact sequence in~\eqref{inizio} with ${\cal O}_X (C + E_1 + E_2)$, and applying equalities~\eqref{hzero} and~\eqref{hu} we find \begin{eqnarray}\label{newhzero} h^0 ({\cal O}_X (C + E_1 + E_2)) = 2\end{eqnarray} and \begin{eqnarray}\label{newhu} h^1 ({\cal O}_X (C + E_1 + E_2)) = 0.\end{eqnarray} Finally, we consider $$0 \to {\cal O}_X (C + E_1 + E_2) \to {\cal O}_X (H) \to {\cal O}_{E_3} (H) \to 0.$$ Equality~\eqref{newhu} states that the induced sequence on global sections $$0 \to H^0({\cal O}_X (C + E_1 + E_2)) \to H^0 ({\cal O}_X (H)) \to H^0 ({\cal O}_{E_3} (H)) \to 0$$ is still exact. The surjectivity of $H^0 ({\cal O}_X (H)) \to H^0 ({\cal O}_{E_3} (H)) = H^0 ({\cal O}_{\mathbb{P}^1} (1))$ shows that $H$ has no base points on $E_3$, and by equality~\eqref{newhzero} we find: $$h^0 ({\cal O}_X (H - E_3)) = h^0 ({\cal O}_X (C + E_1 + E_2)) = 2 < 4.$$ This means that $E_3$ is not a base component of $|H|$. By a symmetric argument, the same is true for $E_1, E_2$.\\ Since all the base locus of $|H|$ must be contained in the curve $C + E_1 + E_2 + E_3$, this proves that $|H|$ has no base points or base components at all.\\
Thus $|H|$ induces a regular map $f : X \to \mathbb{P}^3$, and the self - intersection $$H^2 = 5$$ implies that $f$ is birational on a quintic surface.\\
ii) The equality  $$h^0 ({\cal O}_X (H - C)) = 1$$ shows that the image of $C$ spans a hyperplane. Together with the equality $$H C = 2,$$ this means that $f : C \to f(C)$ is an isomorphism on a smooth plane conic.\\
The equality $$H E_i = 1$$ proves that each $E_i$ is sent isomorphically on a line in $\mathbb{P}^3$.\\
The equality $H = C + E_1 + E_2 + E_3$ implies that there exists a plane in $\mathbb{P}^3$ containing all the images of $C, E_1, E_2, E_3$.\\
 Finally, the long exact sequence induced by $$0 \to {\cal O}_X (E_3) \to {\cal O}_X (C + E_3) \to {\cal O}_C (C + E_3) \to 0$$ proves that $h^0 ({\cal O}_X (C + E_3)) = 1$. But this is the same to say $h^0 ({\cal O}_X (H - E_1 - E_2)) = 1$, so there exists a unique plane containing $f (E_1) + f(E_2)$, which means that $f(E_1) \ne f (E_2)$.\\\\
iii) Let $${\cal E}_i := E_i - K_X.$$ By construction, we have \begin{eqnarray}\label{dec} H = - 2 K_X + E_1 + E_2 + E_3 = {\cal E}_i + {\cal E}_j + E_k\end{eqnarray} for all permutations $(i, j, k)$ of indices $(1, 2, 3)$.\
By Proposition~\ref{basic} we know that $K_X^2 = -1$, so we find the intersection products: \begin{eqnarray}\label{sdlines} {\cal E}_i {\cal E}_j = {\cal E}_i E_j = 1 - \delta_{i, j}.\end{eqnarray} Still by Proposition~\ref{basic}, the short exact sequence $$0 \to {\cal O}_X (- K_X) \to {\cal O}_X ({\cal E}_i) \to {\cal O}_{E_i} ({\cal E}_i) \to 0,$$ induces isomorphisms \begin{eqnarray}\label{hzeroellp} H^0 ({\cal O}_X ({\cal E}_i)) \simeq H^0 ({\cal O}_{E_i} ({\cal E}_i)) \simeq H^0 ({\cal O}_{\mathbb{P}^1}) = \complex\end{eqnarray} and \begin{eqnarray}\label{hunoellp} H^1 ({\cal O}_X ({\cal E}_i)) = 0.\end{eqnarray} This means that each ${\cal E}_i$ is effective and rigid. We want to compute $H^0 ({\cal O}_X (H - {\cal E}_1))$ to calculate how many hyperplanes in $\mathbb{P}^3$ contain the image of ${\cal E}_1$. The decomposition~\eqref{dec} gives a the short exact sequence $$0 \to {\cal O}_X ({\cal E}_3) \to {\cal O}_X (H - {\cal E}_1) \to {\cal O}_{E_2} (H - {\cal E}_1) \simeq {\cal O}_{\mathbb{P}^1} \to 0,$$ where the last isomorphism comes from part $ii)$ and equality~\eqref{sdlines}. Applying equalities~\eqref{hzeroellp} and~\eqref{hunoellp}, we find $$h^0 ({\cal O}_X (H - {\cal E}_1)) = 2,$$ so there exists a pencil of hyperplanes containing $f ({\cal E}_1)$. This means that $f ({\cal E}_1)$ is a line, and its multiplicity for $f(X)$ is easily computed as $$H {\cal E}_1 = (- 2 K_X + E_1 + E_2 + E_3) (E_1 - K_X) = 2.$$ The same is true of course for ${\cal E}_2, {\cal E}_3$.
We prove now that these are distinct lines: indeed, the decomposition~\eqref{dec} gives: $$|H - {\cal E}_1 - {\cal E}_2| = |E_3| = \{E_3\},$$ so only one hyperplane contains $f ({\cal E}_1) + f ({\cal E}_2)$.\\
%Finally, by Proposition~\ref{basic} the self - intersection ${\cal E}_i^2$ equals: $${\cal E}_i^2 = (E_i - K_X)^2 = -1 - 1 + 2 = 0.$$
\\iv) The rigidity of $E_1 + {\cal E}_1$ comes from the equality~\eqref{sdlines}. Indeed, the intersection product $$E_1 (E_1 + {\cal E}_1) = -1 < 0$$ shows that all members of $|E_1 + {\cal E}_1|$ must contain $E_1$, thus $$|E_1 + {\cal E}_1| = E_1 + |{\cal E}_1| = \{E_1 + {\cal E}_1\}.$$\\
v) If $R$ is irreducible and contracted by $f$, then $H R = 0$. By part $i)$ $|H|$ is big, so $R^2 < 0$ by Hodge Index Theorem, see~\cite{TieLuo1990}. Using adjunction formula (see~\cite{EnriquesII2025}, Lemma 9.1.1), we get: $$0 > R^2 = 2 p_a (R) - 2 + \frac{1}{2} R C \ge -2.$$ This proves that $p_a (R) = 0,$ so that $R$ is rational and smooth, and also that $$R^2 \in \{-1, -2\}.$$ If $R$ is a $(-1)$-curve, then $R C = 2$ still by adjunction formula. Moreover $R$ is not one of the $E_i$ by part $ii)$, and thus $$0 = H R = 2 + R (E_1 + E_2 + E_3) \ge 2,$$ a contradiction. Hence $R$ is a smooth rational $(-2)$-curve.\\
vi) We know by the part $ii)$ that the image $f (C)$ is a smooth conic contained in a plane $\Pi$. In this plane there are also the three lines $f (E_i)$. The fact that the restrictions
$e_i := E_i \cdot C$ belong to the same $g_2^1$ of $C$ is thus equivalent to state that these three lines are concurrent at a point $p \in \Pi$.\\
By equality~\eqref{sdlines}, we have ${\cal E}_i E_j = 1 > 0$ for all $i \ne j$, hence the three double lines $f ({\cal E}_i)$ are all concurrent at $p$ as well. In particular, there exists a plane $\Pi'$ which contains both $f (E_1), f ({\cal E}_1)$. This corresponds to an effective divisor $R$ such that $$H = R + E_1 + {\cal E}_1.$$ We use parts $i), ii), iii)$ to compute the intersection $H R$: $$H R = H (H - E_1 - {\cal E}_1) = 5 - 1 - 2 = 2.$$ The image $f (R)$ cannot be a double line $2 l$, otherwise the hyperplanes through $l$ would define a pencil of divisors in the linear system $|H - R| = |E_1 + {\cal E}_1|$, which is forbidden by part $iv)$.\\
Thus the image of $R$ in $\Pi'$ is either a smooth plane conic or the union of two distinct lines. If some component of $R$ is contracted by $f$ we are done by part $v)$, so we can assume that each component of $R$ has $1$-dimensional image. This leaves only two cases: either $R \simeq \mathbb{P}^1$ is isomorphic to a smooth plane conic of $\Pi'$, or $R = A + B,$ where $A, B$ are isomorphic to two distinct lines.\\ In the first case, we just compute the self - intersection $R^2$ using parts $i), ii)$ and $iii)$: $$R^2 = (H - E_1 - {\cal E}_1)^2 = 5 - 1 - 2 H E_1 - 2 H {\cal E}_1 = -2,$$ as required.\\ If $R = A + B$ splits as a sum of two copies of $\mathbb{P}^1$, we want to show that $A^2 = B^2 = -2$. First step is to note that $R$ cannot move. Indeed, $$|R| = |H - E_1 - {\cal E}_1|,$$ and the latter linear system contains just one element, since $\Pi'$ is the only plane containing both $f (E_1), f ({\cal E}_1)$. As a consequence, neither $A$ or $B$ can move. Hence the short exact sequence $$0 \to H^0({\cal O}_X) \to H^0({\cal O}_X (A)) \to H^0({\cal O}_A (A)) \to 0$$ gives $$H^0 ({\cal O}_A (A)) = 0.$$ Since $A$ is isomorphic to $\mathbb{P}^1$, we have: $$H^0 ({\cal O}_A (A)) = H^0 ({\cal O}_{\mathbb{P}^1} (A^2))$$ and thus $$A^2 < 0.$$ At this point we argue as in part $v)$: the self - intersection $A^2$ can only take values in $\{-1, -2\}$. If $A^2 = -1$ then $$H A = A (-2 K_X) + A (E_1 + E_2 + E_3) \ge 2,$$ contradicting the fact that $A \to f (A)$ is an isomorphism on a line. Hence $A^2 = -2$, and the same holds for $B$. $\square$\\\\
The following result is not strictly part of this article. Nonetheless, it is a direct consequence of the previous Lemma, see also the author's doctoral thesis~\cite{Pieroni2025}. It contributes to the geometric description of the quintic model $f(X)$ of a Coble surface $X$. We recall that a Coble surface $X$ is said to be nodal if it contains a curve of self - intersection $-2$ and biregular to $\mathbb P^1$.\\
Given an unnodal Coble surface $X$ with irreducible boundary $C$, and three disjoint $(-1)$-curves $E_1, E_2, E_3$, the linear system $$H = C + E_1 + E_2 + E_3$$ defines a birational morphism of $X$ on a singular quintic surface $\overline X \subset \mathbb{P}^3$. 
\begin{corollary}\label{quintic}
If $X$ is unnodal, the surface $\overline X$ contains a tetrahedron $T$. A vertex $p_0$ of $T$ is a triple point of $\overline X$, and the three edges of $T$ concurrent at $p_0$ are double edges. The other edges form a triangle, and are simple edges of $\overline X$.
\end{corollary}
{\it Proof:} We already know that $|H|$ defines a regular, birational morphism onto a quintic surface $\overline X \subset \mathbb{P}^3$.
The lines $f (E_1), f (E_2), f(E_3)$ are distinct and coplanar. If they were linearly dependent, then $X$ would be nodal by part $vi)$ of Lemma~\ref{propzero}, against the hypothesis.
Then they form a triangle. We also know that $E_i {\cal E}_j = 1$ for all $i \ne j$, hence the simple line $f (E_i)$ must meet the double line $f ({\cal E}_j)$. This means that each $f ({\cal E}_i)$ passes through a vertex of the triangle $f (E_1) + f (E_2) + f (E_3)$. Moreover, the decomposition~\eqref{dec}: $$H = {\cal E}_i + {\cal E}_j + E_k$$ shows that $f ({\cal E}_i), f({\cal E}_j)$ are pairwise coplanar, for all permutations $(i, j, k)$. This implies that the double lines $f ({\cal E}_i)$ share a common point $p_0$, which is the fourth vertex of the required tetrahedron $T$.
In a suitable system $[X_0, X_1, X_2, X_3]$ of projective coordinates in $\mathbb{P}^3$, we can assume that $$p_0 = [1, 0, 0, 0]$$ and that the remaining three vertices are: $$[0, 1, 0, 0], [0, 0, 1, 0], [0, 0, 0, 1].$$ A quintic surface $\overline X$, with  prescribed multiplicities as above along the edges of the tetrahedron $T$, is defined by a general equation as follows: $$\alpha X_0 X_1^2 X_2^2 + \beta X_0 X_1^2 X_3^2 + \gamma X_0 X_2^2 X_3^2 + X_1 X_2 X_3 q (X_0, X_1, X_2, X_3) = 0,$$ where $\alpha, \beta, \gamma \in \complex$ and $q$ is a quadric in four projective variables. This shows that $p_0$ is a triple point for $\overline X$ and it concludes the proof. $\square$
\begin{remark}\label{moduli}
The author proved in his Ph. D. thesis that a partial converse to Corollary~\ref{quintic} is also true. Namely, we showed that, given a quintic surface $\overline X \subset \mathbb{P}^3$ defined by a generic expression as above, if the normalization $X^{\nu}$ is smooth, then it is a Coble surface.\\
Note that these quintics depend on $13$ parameters, up to rescaling, thus they move in an open subset $U$ of a space $\mathbb{P}^{12}$. Moreover, these expressions are preserved by the linear group $G$ of change of variables: $$X_0' = \lambda_0 X_0,$$ $$X_1' = \lambda_1 X_i,$$ $$X_2' = \lambda_2 X_j,$$ $$X_3' = \lambda_3 X_k,$$ where $(i, j, k)$ is a permutation of the indices $(1, 2, 3)$. The projectivized group $\mathbb{P} (G)$ has dimension $3$, thus the quotient $$U // \mathbb{P} (G)$$ has dimension $9$, as well as the moduli space of Coble surfaces.
\end{remark}

\newpage
\section{The main result}
In~\cite{Pompilj1940}, Pompilj considered the following construction: %let $\overline C \subset \mathbb{P}^2$ be an irreducible sextic curve with $10$ double points $p_1, \ldots, p_{10}$. The blow - up $$X := Bl_{p_1, \ldots, p_{10}} \mathbb{P}^2$$ is a Coble surface. With the decomposition $$Pic (X) = \integer L \oplus \bigoplus_{i = 1}^{10} \integer E_i,$$ consider the linear system $${\cal P}_1 := 6 L - 2 E_2 - \cdots - 2 E_{10}.$$ This has no base points, since it is generated by two disjoint curves, namely $C + 2 E_1$ and $2 {\cal E}_1$, . Using the adjunction formula one shows that $$p_a ({\cal P}_1) = 1,$$ thus the generic member of ${\cal P}_1$ is a smooth elliptic curve. 

\begin{definition}\label{def}
Let $p_1, \ldots, p_{10}$ be nodes for an irreducible sextic $\overline C \subset \mathbb{P}^2$. For each $k = 1, 2, 3$, consider the Del Pezzo surface $$Y_k := Bl_{p_k, p_4, \ldots, p_{10}} \mathbb{P}^2$$ of degree $1$, with its Bertini involution $$i_k : Y_k \to Y_k.$$ Consider also the Coble surface $$X := Bl_{p_1, \ldots, p_{10}} \mathbb{P}^2,$$ whose boundary curve $C \subset X$ is the strict transform of $\overline C$.\\
We know that $$Pic (X) = \integer L \oplus \bigoplus_{i = 1}^{10} \integer E_i.$$ With this decomposition, consider for each index $j = 1, 2, 3$ the linear system $${\cal P}_j := 6 L - 2 E_1 - \cdots - 2 \check{E_j} - \cdots - 2 E_{10}$$ of strict transforms of sextics which are nodal at $p_s$, for all $s \ne j$.
\end{definition}
Each ${\cal P}_j$ is a pencil of curves without base points, since it is generated by two disjoint sections, namely $C + 2 E_j$ and $2 {\cal E}_j$, where $${\cal E}_j := 3 L - E_1 \cdots - \check{E}_j - \cdots - E_{10}$$ is the strict transform of the unique plane cubic passing through all $p_s$, with $s \ne j$. Moreover, each $H \in |{\cal P}_j|$ has arithmetic genus $$p_a (H) = 1,$$ since it is the strict transform of a plane sextic with $9$ double points. Thus the generic fibers of ${\cal P}_j$ are smooth elliptic curves. Note that all $E_i$ with $i \ne j$ are bisections for ${\cal P}_j$. On a smooth member $H \in {\cal P}_1$, Pompilj considers the automorphism $T_1 : H \to H$ defined by the following condition: $$\mathcal O_H(T_1 (x) - x ) \simeq  {\cal O}_H (E_2 - E_3) \quad \forall x \in H. $$ 
In a similar way, for a smooth fiber $H \in |{\cal P}_2|$, Pompilj defined $$T_2 (x) - x \simeq_H {\cal O}_H (E_3 - E_1) \quad \forall x \in H,$$ and for $H \in |{\cal P}_3|$ he set: $$T_3 (x) - x \simeq_F {\cal O}_H (E_1 - E_2) \quad \forall x \in H.$$
Pompilj claimed that $$(T_1 \circ T_2 \circ T_3)|_C = \mathbb{1}_C.$$ Coble showed that this happens only in a $1$-codimensional family of Coble surfaces. Our claim is to show that this is a family of nodal Coble surfaces.

\begin{corollary}[\cite{Coble1939, Dolg2019}]\label{lift}
Let $p_1, \ldots, p_{10}$ be nodes for an irreducible sextic $\overline C \subset \mathbb{P}^2$. Using the notations of Definition~\ref{def}, for each permutation $(i, j, k)$ of indices $(1, 2, 3)$ we have:\\
i) The Bertini involution $i_k : Y_k \to Y_k$ lifts to the Coble surface $X = Bl_{p_i, p_j} Y_k$ as an involution $i_k' : X \to X$ which preserves the curve $C$.\\
ii) The action of $i_k'$ on $C \simeq \mathbb{P}^1$ is associated to the $g_2^1$ generated by $(E_i)|_C, (E_j)_C$.\\
iii) The lifted involution $i_k'$ preserves each fiber $H$ of $|{\cal P}_j|$. If $H$ is also smooth, then $$x + i_k' (x) \in |{\cal O}_H (E_i)| \quad \forall x \in H.$$
\end{corollary}
{\it Proof:} In all the proof, the triple $(i, j, k)$ will be a permutation of $(1, 2, 3)$. We will denote by $$\tau_k : X \to Y_k$$ the blow down of $E_i, E_j$.\\
$i)$ Consider the strict transform $C \subset X$ of $\overline C$. We have $$C = - 2 K_X = -2 \tau_k^* K_{Y_k} - 2 E_i - 2 E_j,$$ so that $\tau_k (C)$ is an irreducible member of $|- 2 K_{Y_k}|$ with double points at $p_i, p_j$. By the Proposition~\ref{fix}, the Bertini involution $i_k$ fixes both $p_i, p_j$, hence it lifts to a biregular involution $i_k'$ on $X = Bl_{p_i, p_j} Y_k$. Moreover, $i_k$ preserves each member of $|- 2 K_{Y_k}|$, so in particular it preserves $\tau_k (C)$. This implies that the strict transform $C$ is preserved by the lift $i_k'$.\\\\
$ii)$ By Proposition~\ref{fix}, $i_k$ swaps the tangent directions of $\tau_k (C)$ at both $p_i, p_j$. Consequently, the lift $i_k'$ swaps the two points in each pair $C \cap E_i, C\cap E_j$. This concludes the proof, since any involution $i$ on $\mathbb{P}^1 \simeq C$ is determined by the $g_2^1$ of the pairs $x + i (x), x \in \mathbb{P}^1$.\\\\
$iii)$ If $H \in |{\cal P}_j|$, then its class in $X$ is given by $$H = -2 \tau_k^* K_{Y_k} - 2 E_i.$$ Then $H$ is the strict transform of an element $\tau_k (H)$ in $|- 2 K_{Y_k}|$ with a double point at $p_i$. The Bertini involution $i_k$ preserves each member of $|- 2 K_{Y_k}|$, so in particular it preserves $\tau_k (H)$. Consequently, the lift $i_k'$ preserves $H$.\\
Assume now that $H$ is smooth. Then $\tau_k (H)$ has $p_i$ as its unique singular point. This is obtained attaching together the two points in $H \cap E_i$, which correspond to the tangent directions of $\tau_k (H)$ at $p_i$. These tangent directions are swapped by $i_k$ by the Proposition~\ref{fix}, hence $i_k'$ swaps the two points $H \cap E_i$. Consequently, there exists an isomorphism: $$\frac{H}{i_k'} \simeq \frac{\tau_k (H)}{i_k}.$$ The quotient on the right side is known, since $\tau_k (H)$ is an irreducible member of $|- 2 K_{Y_k}|$, quotiented by the Bertini involution $i_k$. This corresponds to a smooth hyperplane section of the quotient quadric cone $\frac{Y_k}{i_k} \subset \mathbb{P}^3$, which is just a copy of $\mathbb{P}^1$. Hence $$\frac{H}{i_k'} \simeq \mathbb{P}^1.$$ This means that all the pairs $x + i_k' (x)$ are linearly equivalent one each other, as $x$ moves in $H$. But one of this pairs is exactly $H \cap E_i = {\cal O}_H (E_i)$, thus $$x + i_k' (x) \in |{\cal O}_H (E_i)| \quad \forall x \in H, \quad \forall H \in |{\cal P}_j| {\rm\, smooth}.$$ $\square$\\\\
\begin{corollary}[\cite{Coble1939, Dolg2019}]\label{comp}
The three birational morphisms $T_1, T_2, T_3 : X \dasharrow X$ are the compositions $$T_1 = i_3' \circ i_2', \quad T_2 = i_1' \circ i_3', \quad T_3 = i_2' \circ i_1'.$$ In particular, $T_1, T_2, T_3$ extend to biregular automorphism of $X$ into itself.
\end{corollary}
{\it Proof:} We prove just the first equality. By construction, $T_1$ acts fiberwise on each smooth member of $|{\cal P}_1|$. By Corollary~\ref{lift}, both $i_2', i_3'$ preserve each fiber of this pencil. Since the union of smooth fibers is dense in $X$, it suffices to show that the equality $T_1 = i_3' \circ i_2'$ holds on each smooth fiber $H$. Still by Corollary~\ref{lift}, we know that $$x + i_2' (x) \in |{\cal O}_H (E_3)| \quad \forall x \in H$$ and $$x + i_3' (x) \in |{\cal O}_H (E_2)| \quad \forall x \in H.$$ By definition of $T_1$, for any $x \in H$ we have: $$T_1 (x) \simeq {\cal O}_H (E_2 - E_3 + x) =$$ $$= {\cal O}_H (E_2 - (E_3 - x)) = {\cal O}_H (E_2 - i_2' (x)) \simeq i_3' (i_2' (x)).$$ Since $H$ has genus $1$, this forces $$T_1 (x) = i_3' (i_2' (x)).$$
Finally, Corollary~\ref{lift} states that all the $i_k'$'s are biregular on $X$, thus all the $T_j$'s extend to biregular automorphisms of $X$. $\square$\\
We need the following lemma before the final statement.
\begin{lemma}\label{codtre}
Let $\sigma_k : C \to C$ be the restrictions $\sigma_k := (i_k')|_C, k = 1, 2, 3$. Let $H := C + E_1 + E_2 + E_3$ be the polarization considered in Proposition~\ref{propzero}. Assume that $(E_1)|_C, (E_2)|_C, (E_3)|_C$ are linearly independent divisors. By Corollary~\ref{quintic}, the image surface $\overline X \subset \mathbb{P}^3$ is a quintic with equation \begin{eqnarray}\label{Cobleq} \alpha X_0 X_1^2 X_2^2 + \beta X_0 X_1^2 X_3^2 + \gamma X_0 X_2^2 X_3^2 + X_1 X_2 X_3 q = 0,\end{eqnarray} with $\alpha, \beta, \gamma \in \complex \setminus 0$, and $q = q (X_0, X_1, X_2, X_3)$ is a quadric form.\\ Assume that $$\sigma_3 \circ \sigma_2 \circ \sigma_1 = \mathbb{1}_C.$$ Then the quadric $q$ does not contain the terms $X_1 X_2, X_1 X_3, X_2 X_3$.\\
In particular, Coble surfaces with this property form a $3$-codimensional family in the moduli space of Coble surfaces.
\end{lemma}
{\it Proof:} Let $$\hat q (X_1, X_2, X_3) := q (0, X_1, X_2, X_3).$$ The plane $\{X_0 = 0\}$ cuts on the singular surface $\overline X$ the reducible curve $$X_1 X_2 X_3 \hat q = 0,$$ where the plane conic $V (\hat q)$ is the isomorphic image of $C$, while each line $\{X_i = 0\}$ is the image of $E_i$. By Corollary~\ref{lift}, the action of $\sigma_1$ on $C$ is associated to the $g_2^1$ generated by $(E_2)|_C, (E_3)|_C$. In the plane $\{X_0 = 0\}$, this $g_2^1$ is cut by the pencil of lines through the point $[0, 1, 0, 0]$. The pair of fixed points of $\sigma_1$ is cut by the polar line to $V (\hat q)$ with center this point, which has equation $$\frac{\partial \hat q}{\partial X_1} = 0.$$ On the other hand, this line must coincide with $\{X_1 = 0\}$ by Proposition~\ref{casodue}, hence there exists a constant $\lambda_1 \ne 0$ such that $$\frac{\partial \hat q}{\partial X_1} = \lambda_1 X_1.$$ Similarly, there are $\lambda_2, \lambda_3 \in \complex \setminus 0$ such that $$\frac{\partial \hat q}{\partial X_2} = \lambda_2 X_2, \quad \frac{\partial \hat q}{\partial X_3} = \lambda_3 X_3.$$ This means that $$\hat q = \frac{1}{2} (\lambda_1 X_1^2 + \lambda_2 X_2^2 + \lambda_3 X_3^2),$$ thus $q$ does not contain the terms $X_i X_j$, with $1 \le i < j \le 3$.\\
For the last part of the statement, we mentioned in Remark~\ref{moduli} that the expressions of the form~\eqref{Cobleq} form a projective space $\mathbb{P}^{12},$ and that there exists an open subset $U \subset \mathbb{P}^{12}$ of surfaces whose normalization is actually a Coble surface. The vanishing of these three coefficients describes a $3$-codimensional linear subspace $\Lambda \subset \mathbb{P}^{12},$ hence we are interested in the intersection $U \cap \Lambda$. Consider the group $G$ of linear transformations: $$X_0' = \lambda_0 X_0,$$ $$X_1' = \lambda_1 X_i,$$ $$X_2' = \lambda_2 X_j,$$ $$X_3' = \lambda_3 X_k,$$ where $(i, j, k)$ is a permutation of the indices $(1, 2, 3).$ The projectivized $\mathbb{P} (G)$ has dimension $3$, and it preserves the space $\mathbb{P}^2$, the open subset $U$ and the linear subspace $\Lambda$. Thus we find an inclusion $$(U \cap \Lambda) // \mathbb{P} (G) \subset U // \mathbb{P} (G),$$ which is the inclusion of a locally closed subset of dimension $6$ in a family of dimension $9$. $\square$

\begin{theorem}\label{final}
Assume that $$(T_1 \circ T_2 \circ T_3)|_C = \mathbb{1}_C.$$
Then there are two possibilities:\\
i) $X$ contains $(-2)$-curves.\\
Moreover, in this case we have $$T_1 = T_2 = T_3 = \mathbb{1}_X.$$
ii) $X$ belongs to a $3$-codimensional family of Coble surfaces.
\end{theorem}
{\it Proof:} By Corollary~\ref{comp}, this is equivalent to $${(i_3' \circ i_2' \circ i_1')|_C}^2 = \mathbb{1}_C.$$ Let $E_1, E_2, E_3 \subset X$ be the exceptional divisors associated to $p_1, p_2, p_3$. Through the identification $C := \mathbb{P}^1$, the restricted divisors $(E_i)|_C$ are respectively defined by three quadratic polynomials $A_1, A_2, A_3 \in H^0 ({\cal O}_{\mathbb{P}^1} (2))$.\\
By Corollary~\ref{lift}, there exists a well - defined restriction $$\sigma_k := (i_k')|_C : C \to C,$$
corresponding to the $g_2^1$ on $C \simeq \mathbb{P}^1$ generated by $(E_i)|_C, (E_j)|_C$, for all permutations $(i, j, k)$ of $(1, 2, 3)$.\\ 
As a consequence, Proposition~\ref{propdue} states that the pairs of fixed points of $\sigma_1, \sigma_2, \sigma_3$ are respectively the Jacobian forms $$F_1 = J (A_2, A_3),$$
$$F_2 = J (A_1, A_3),$$ 
$$F_3 = J (A_1, A_2).$$
If we assume that $$\sigma_3 \circ \sigma_2 \circ \sigma_1 \ne \mathbb{1}_C,$$ we show that the case $i)$ of the statement is realized.\\
By Proposition~\ref{propuno}, the $F_i$'s are linearly dependent polynomials in $H^0 ({\cal O}_{\mathbb{P}^1} (2))$. If we write down explicitly $A_1, A_2, A_3 \in H^0 ({\cal O}_{\mathbb{P}^1} (2))$, say $$A_i = a_{i, 1} U^2 + a_{i, 2} U V + a_{i, 3} V^2,$$ then: 
$$F_1 = 2 (a_{2, 1} a_{3, 2} - a_{2, 2} a_{3, 1}) U^2 + 4 (a_{2, 1} a_{3, 3} - a_{2, 3} a_{3, 1}) U V + 2 (a_{2, 2} a_{3, 3} - a_{2, 3} a_{3, 2}) V^2,$$
$$F_2 = 2 (a_{1, 1} a_{3, 2} - a_{1, 2} a_{3, 1}) U^2 + 4 (a_{1, 1} a_{3, 3} - a_{1, 3} a_{3, 1}) U V + 2 (a_{1, 2} a_{3, 3} - a_{1, 3} a_{3, 2}) V^2,$$
$$F_3 = 2 (a_{1, 1} a_{2, 2} - a_{1, 2} a_{2, 1}) U^2 + 4 (a_{1, 1} a_{2, 3} - a_{1, 3} a_{2, 1}) U V + 2 (a_{1, 2} a_{2, 3} - a_{1, 3} a_{2, 2}) V^2.$$ 
Let $A$ be the matrix whose entries are the coefficient of the $A_i$'s with respect to the standard basis $U^2, U V, V^2$ of $H^0 ({\cal O}_{\mathbb{P}^1} (2))$, and let $F$ be the matrix whose entries are the coefficient of the $F_i$'s which we computed above. Then an explicit calculation proves that, up to a scalar factor in $\complex$, we have: $${\rm det\,} F = ({\rm det\,} A)^2.$$ Since the $F_i$'s are linearly dependent, the $A_i$'s are linearly dependent too. This means that the restrictions on $C$ of exceptional divisors $E_1, E_2, E_3$ belong to the same $g_2^1$, and by part $vi)$ of Lemma~\ref{propzero} we find that $X$ contains $(-2)$-curves.\\
Now we prove the second part of the statement: part $vi)$ of Lemma~\ref{propzero} states that the three divisors $R_i := H - E_i - {\cal E}_i$ are effective. Writing them down explicitly, we get $$R_i = 3 L - 2 E_i - E_4 - \cdots - E_{10}, \quad i = 1, 2, 3.$$ Let $Z$ be the Del Pezzo surface of degree $2$ given by $$Z := Bl_{p_4, \ldots, p_{10}} \mathbb{P}^2.$$ The effectiveness of the $R_i$'s gives three members of $|- K_Z|$ which are singular at $p_1, p_2, p_3$ respectively. By Proposition~\ref{fixgeiser} we know that $p_1, p_2, p_3$ are fixed by the Geiser involution $i_G : Z \to Z$, and by Proposition~\ref{liftgeiser} we know then that the Bertini involutions $i_k$ on the Del Pezzo surfaces $Y_k = Bl_{p_k} Z$ are all lifts of $i_G$. By definition $i_k' : X \to X$ is the lift of $i_k$, thus $i_1' = i_2' = i_3' : X \to X$ all coincide with the total lift of $i_G$. Finally, applying Corollary~\ref{comp}, we get $T_1 = T_2 = T_3 = \mathbb{1}_X$.\\
This completes the proof of case $i)$. For case $ii)$, we can thus assume that $$\sigma_3 \circ \sigma_2 \circ \sigma_1 = \mathbb{1}_C$$ and also that the restrictions $$(E_1)|_C, (E_2)|_C, (E_3)|_C$$ are linearly independent. But in this case we just need to apply the previous Lemma~\ref{codtre}. $\square$

 %\newpage
\phantomsection
\addcontentsline{toc}{section}{References}
\bibliographystyle{plain}
%\bibliography{REF}

\end{document}